\documentclass[a4paper]{article}
\usepackage{a4wide}
\usepackage[utf8]{inputenc}
\usepackage[english]{babel}
\usepackage{amssymb}
\usepackage{amsfonts}
\usepackage{amsmath}
\usepackage{amsthm}
\usepackage{graphicx}
\usepackage{hyperref}

\newtheorem{theorem}{Theorem}[section]

\newtheorem{proposition}{Proposition}[section]
\theoremstyle{remark}
\newtheorem{remark}{Remark}[section]
\theoremstyle{definition}

\begin{document}

\begin{center}
\LARGE{Branching random walks on $\mathbb{Z}$ with one particle generation center and symmetrically located absorbing sources}
\end{center}
\begin{center}
\textit{E.\,Filichkina$^{1}$, E.\,Yarovaya$^1$}
\end{center}

\begin{center}
$^1$ Department of Probability Theory, Lomonosov Moscow State University,\\
Moscow, Russia\\
\end{center}
\medskip

\begin{abstract}
We consider a time-continuous branching random walk (BRW) on a one-dimensional lattice on which there is one center (lattice point) of particle generation, called branching source. The generation of particles in the branching source is described by a Markov branching process. Some number (finite or infinite, depending on the problem formulation) of absorbing sources is located symmetrically around the branching source. For such configurations of sources we obtain necessary and sufficient conditions for the existence of an isolated positive eigenvalue of the evolution operator. It is shown that, under some additional assumptions, the existence of such an eigenvalue leads to an exponential growth of the number of particles in each point of the lattice.

\medskip

\end{abstract}
\sloppy
\section{Introduction}\label{sec1}

The random walk underlying the BRW is given by the transition intensity matrix ${A = \left(a(x,y)\right)_{x,y \in \mathbb{Z}}}$. We will consider the case of simple symmetric random walk, when the particle transition is possible only to neighboring points of the lattice. For such a random walk the elements of the matrix $A$ have the following form:

$$a(x,y) =
  \begin{cases}
    \hfill\frac{\varkappa}{2}, \quad & \text{if } |x-y|=1, \\
    -\varkappa, \quad                & \text{if } x=y,     \\
    \hfill0, \quad                   & \text{if } |x-y|>1,
  \end{cases}$$
where $\varkappa > 0$ is some parameter characterizing the walking intensity.

Let $p(t,x,y)$ be the transient probability of random walk, that is, the probability that at time $t \geq 0$ the particle is at point $y,$ provided that at time $t=0$ it was at point~$x$.  Asymptotically, when $h \rightarrow 0$, the transition probabilities are expressed through the transition intensities as follows:
\begin{align*}
  p(h,x,y) & = a(x,y)h + o(h), \quad x \neq y, \\
  p(h,x,x) & = 1 + a(x,x)h + o(h).
\end{align*}

As noted, for example, in~\cite{book}, the transition probabilities $p(t,x,y)$ are the solution of the Cauchy problem:
$$\partial_t p(t,x,y) = \varkappa \Delta p(t,x,y), \quad p(t, \cdot, y) = \delta_y(\cdot),$$
where $\delta(\cdot)$ is Kronecker delta on $\mathbb{Z}$ and $\varkappa\Delta: l^p(\mathbb{Z}) \rightarrow l^p(\mathbb{Z})$, $1 \leq p \leq \infty$, is the walking operator, acting by the following formula:
$$\varkappa\Delta\psi(x) = \frac{\varkappa}{2} (\psi(x+1) + \psi(x-1)) - \varkappa\psi(x).$$

The branching process at the particle generation center is given by an infinitesimal generating function $$f(u) = \sum_{n = 0}^{\infty} b_{n} u^n,\, 0 \leq u \leq 1,$$ where $b_{n} \geq 0$ for $n \neq 1,\, b_{1} < 0,\, \sum_{n = 0}^{\infty}b_{n} = 0$.
The coefficients $b_{n}$ determine the principal linear part of the probability $p_*(h,n)$  of having $n$ particles at time $h$ given that there was one particle at the initial time $t=0$.
The coefficients $b_{n}$ for $n \geq 1$ can be interpreted as intensities of appearance of $n$ descendants of the particle, including the particle itself, while $b_{0}$ can be interpreted as intensities of death, or absorption, of the particle.

For absorbing sources, the generating function has the simpler form $$\overline{f}(u) = b_0 + \overline{b}_1 u = b_0 - b_0 u = b_0(1-u).$$
The death intensity $b_0$ is assumed to be independent of the lattice point and equal in all branching sources.

It is convenient for describing the behavior of the process to introduce the parameter $\beta := \left.f'(u)\right|_{u=1}$, which characterizes the intensity of the source. Besides it, we will also consider the parameter $\beta^*:=\sum_{n>1} (n-1)b_n$. This parameter characterizes the intensity of ``pure'' reproduction in the source and is related to the parameter $\beta$ by the relation $\beta=\beta^*-b_{0}$. Note that when $b_{0}>0$, the parameter $\beta$ can take any real value, while the parameter $\beta^*$ is assumed to be positive.

Consider the BRW on $\mathbb{Z}$ with the configuration of branching sources of the following special form. Without restriction of generality, we will assume that the branching source is located at the point $x = 0$, in which the branching of particles with intensity $\beta$ occurs, and at $2n>0$ points to the right and left ($n$ to the right and $n$ to the left) symmetrically from the branching source only the death of a particle with intensity $b_0$ is possible.

Note that processes with a finite number of branching sources are rather difficult to study~\cite{spr}. In particular, this is due to the fact that there is a dependence on the configuration of the sources~\cite{Y2018}. Therefore, in this paper we consider the BRW with a configuration of branching sources of a special kind. In this case, it is possible to obtain necessary and sufficient conditions for the existence of a positive eigenvalue of the evolution operator of the average particle numbers at each point of the lattice in explicit form.

The evolution of particles in the considered process proceeds as follows: a particle located at some time ${t>0}$ at a point ${x \in \mathbb{Z}},$ in a small time ${dt \rightarrow 0}$ can either jump to a point ${y \neq x}$, ${y \in \mathbb{Z}}$, with probability $a(x,y)dt + o(dt),$ either die with probability $b_0 dt + o(dt)$ (if $x \in \{-n,n\}$)$,$ or produce ${k>1}$ descendants, including itself, with probability ${b_k dt + o(dt)}$ (if point $x$ is the source of branching). Otherwise, with probability $$1 + a(x,x)dt + \delta_0(x)b_1 dt + \sum\limits_{m = -n,\dots,-1,1,\dots,n}\delta_m(x)(-b_0 dt) + o(dt)$$  the particle remains at point $x$ for the entire time interval $[t, t+dt]$. We assume that each new particle evolves according to the same law, independently of the other particles and of the whole prehistory.

The main objects of study in BRWs are the number of particles at the time ${t \geq 0}$ at the point $y \in \mathbb{Z}$ (the local number of particles), denoted by $\mu(t, y),$ the total number of particles (particle population), denoted by ${\mu(t) = \sum_{y \in \mathbb{Z}} \mu(t, y)},$ and their integer moments, which are denoted as ${m_n(t,x,y) := \mathsf{E}_x \mu^n(t,y)}$ and ${m_n(t, x) : = \mathsf{E}_x \mu^n(t)}$, ${n \in \mathbb{N}}$, where $\mathsf{E}_x$ is the expectation under the condition ${\mu(0, y) = \delta_x(y)}$. We will assume that at the initial moment of time $t=0$ the system consists of a single particle located at the point ${x \in \mathbb{Z}}$, so the expectations of the local and total number of particles satisfy the initial conditions: ${m_n(0,\cdot,y) = \delta_y(\cdot)}$ and ${m_n(0, \cdot) \equiv 1}$, respectively.

\section{BRW with a finite number of absorbing sources}\label{sec2}

Note that the process under consideration is a special case of the process that was studied in \cite{spr} for an arbitrary configuration of a finite number of branching sources.

The first moments in the considered process satisfy the differential equation
\begin{equation*}
  \dfrac{d m_1}{dt} = \mathcal{H}_{n}m_1,
\end{equation*}
where the operator $\mathcal{H}_n$ has the following form
\begin{equation*}
  \mathcal{H}_n = \varkappa \Delta + \beta\Delta_0 - b_0\sum\limits_{k=1}^{n} (\Delta_k + \Delta_{-k}).
\end{equation*}
Here the operator $\Delta_k$ is defined by the equality $\Delta_k = \delta_k \delta_k^T$, where $\delta_k = \delta_k(\cdot)$ denotes a column-vector on $\mathbb{Z}$ taking the unit value at the point $k \in \mathbb{Z}$ and vanishing at other points.

In \cite{spr}, a limit theorem was obtained stating that, if there exists an isolated positive eigenvalue in the spectrum of the operator describing the evolution of the mean number of particles, an exponential growth of both the total number of particles and the number of particles at every point of the lattice will be observed in the process.
For the considered case (with the operator $\mathcal{H}_n$) this theorem can be formulated in the following form.

\begin{theorem}[\cite{spr}, Theorem 1 for $\mathcal{H}_n$]\label{thm1}
  Let the operator $\mathcal{H}_n$ have an isolated eigenvalue $\lambda_n > 0$, and let the remaining part of its spectrum be located on the halfline $\{\lambda \in \mathbb{R}: \lambda \leq \lambda_n - \varepsilon \}$, where $\varepsilon > 0$. If $\beta^{(k)} = O(k! k^{k-1}), k \in \mathbb{N}$, then the following statements hold in the sense of convergence in distribution
  \begin{equation*}
    \lim\limits_{t \rightarrow \infty} \mu(t,y) e^{-\lambda_n t} = \xi \psi(y), \quad 	\lim\limits_{t \rightarrow \infty} \mu(t) e^{-\lambda_n t} = \xi,
  \end{equation*}
  where $\psi(y)$ is the eigenfunction corresponding to the eigenvalue $\lambda_n$ and $\xi$ is a nondegenerate random variable.
\end{theorem}

\begin{remark}
  The use of the martingale approach, see \cite{martBRW}, allows us to restrict ourselves to the study of the first two moments to obtain the limit theorem, and to prove a stronger statement about the convergence of the considered random variables to the limit in the mean-square.
  In \cite{martBRW} BRWs with an infinite number of positive intensity sources on $\mathbb{Z}^d$ were considered, for which the relation $\beta(x) \rightarrow 0$ at $||x|| \rightarrow \infty$ holds, where $||\cdot||$ is the norm in $l^2(\mathbb{Z}^d)$.
  We assume that the technique proposed in the \cite{martBRW} will allow us to speak about convergence in the mean-square and for the process considered in this paper. However, this is not the main goal of the paper.
\end{remark}

Thus, the problem of finding the conditions under which there exists an isolated positive eigenvalue in the spectrum of the operator $\mathcal{H}_n$, which is necessary to ensure exponential growth of the particle numbers in the lattice points, turns out to be important.

The eigenvalue problem for this operator is as follows:
\begin{equation} \label{eigv1}
  \mathcal{H}_n f= \varkappa \Delta f + \beta\Delta_0 f - b_0\sum\limits_{k=1}^{n} (\Delta_k + \Delta_{-k}) f =\lambda f.
\end{equation}

Let us apply to the obtained ratio the Fourier transform defined for the function $g(x)$ by the formula
\begin{equation*}
  \hat{g}(\theta) = \sum\limits_{x}g(x)e^{ix\theta}.
\end{equation*}

Note that
\begin{align*}
  \widehat{\varkappa \Delta f}(\theta) & = \varkappa \sum\limits_{x}(\frac{1}{2}f(x-1) + \frac{1}{2}f(x+1) - f(x))e^{ix\theta}                                  \\
                                       & = \varkappa (\frac{1}{2}(e^{i\theta} + e^{-i\theta}) - 1)\hat{f}(\theta) = \varkappa(\cos{\theta} - 1)\hat{f}(\theta).
\end{align*}

After applying the Fourier transform to \eqref{eigv1} we obtain:
\begin{align*}
  \lambda \hat{f}(\theta) & = \varkappa(\cos \theta - 1)\hat{f}(\theta) +         \beta f(0) - b_0 \sum\limits_{k=1}^{n} (f(k)e^{i\theta k} + f(-k)e^{-           i\theta k}) \\
                          & = \varkappa(\cos \theta - 1)\hat{f}(\theta) + \beta f(0) - b_0                \sum\limits_{k=1}^{n} f(k) 	(e^{i\theta k} + e^{-i\theta k}).
\end{align*}
Whence we obtain for $\hat{f}(\theta)$ the expression
\begin{equation} \label{eq2}
  \hat{f}(\theta) = \frac{\beta f(0) - b_0 \sum_{k=1}^{n} f(k)         (e^{i\theta k} + e^{-i\theta k})}{\lambda + \varkappa -  \varkappa \cos     \theta}.
\end{equation}

Note that for any $k \in \mathbb{Z}$ the following relations hold
\begin{equation*}
  \frac{1}{2\pi}\int\limits_{-\pi}^{\pi}\hat{f}(\theta)e^{-ik\theta}\,d\theta = 	\frac{1}{2\pi}\int\limits_{-\pi}^{\pi} \sum\limits_{x}f(x)e^{ix\theta} e^{-ik\theta}\,d\theta = f(k).
\end{equation*}
Integrating the expression \eqref{eq2} multiplied by $e^{-il\theta}$, on the segment $[-\pi,\pi]$, we obtain for ${l = 0,\dots,n}$:
\begin{align*}
  f(l) & = \frac{1}{2\pi}\int\limits_{-\pi}^{\pi}\frac{\beta f(0) - b_0 \sum_{k=1}^{n} f(k) (e^{i\theta k} + e^{-i\theta k})}{\lambda + \varkappa -  \varkappa \cos \theta}	e^{-i\theta l}\,d\theta            \\
       & = \frac{1}{2\pi}\int\limits_{-\pi}^{\pi}\frac{\beta f(0) e^{-i\theta l} - b_0 \sum_{k=1}^{n} f(k) (e^{i\theta (k-l)} + e^{i\theta (-k-l)})}{\lambda + \varkappa -  \varkappa \cos \theta}	\,d\theta	.
\end{align*}

For further analysis, we need the following auxiliary statement.

\begin{proposition}
  Let $a>b>0$, then
  $$
    \int_{-\pi}^{\pi} \frac{\cos{n\theta}}{a-b\cos{\theta}} \,d\theta = 2\pi\frac{z_1^{|n|}}{\sqrt{a^2-b^2}},
  $$
  where $z_1 = \frac{a-\sqrt{a^2-b^2}}{b}$.
\end{proposition}

\begin{proof}
  It is enough to prove the statement for non-negative $n$.

  First note that the following chain of equalities holds:
  \begin{equation*}
    \int\limits_{-\pi}^{\pi} \frac{\cos{n\theta}}{a-b\cos{\theta}} \,d\theta= \int\limits_{-\pi}^{\pi} \frac{e^{in\theta}}{a-b\frac{e^{i\theta} + e^{-i\theta}}{2}} \,d\theta = \int\limits_{-\pi}^{\pi} \frac{2e^{i(n+1)\theta}}{2ae^{i\theta}-be^{2i\theta}-b} \,d\theta.
  \end{equation*}
  Let us substitute the variables $z = e^{i\theta}$, $dz = ie^{i\theta}\,d\theta$ in the latter integral.
  Then we get
  \begin{equation*}
    \int_{-\pi}^{\pi} \frac{\cos{n\theta}}{a-b\cos{\theta}} \,d\theta=\int_{|z|=1} \frac{-2iz^n}{2az-bz^2-b}\,dz.
  \end{equation*}
  The isolated singularities of the integrand in the right-hand part are $z_{1,2} = \frac{a \mp \sqrt{a^2-b^2}}{b}$, while the set $|z|<1$ contains only $z_1 = \frac{a - \sqrt{a^2-b^2}}{b}$.
  Then using Cauchy's residue theorem, we obtain that the integral under consideration is $2\pi\frac{z_1^n}{\sqrt{a^2-b^2}}$.
\end{proof}

In the problem under consideration $a = \lambda + \varkappa$, $b = \varkappa$, $\sqrt{a^2-b^2} = \sqrt{(\lambda + \varkappa)^2-\varkappa^2} = \sqrt{\lambda(\lambda + 2\varkappa)}$.
For convenience, we introduce additional notations:
\begin{equation*}
  \zeta = \frac{\lambda+\varkappa - \sqrt{\lambda(\lambda + 2\varkappa)}}{\varkappa}, \quad r = \sqrt{\lambda(\lambda + 2\varkappa)}.
\end{equation*}
Then
$$
  \int_{-\pi}^{\pi} \frac{\cos{n\theta}}{(\lambda+\varkappa)-\varkappa \cos{\theta}}\,d\theta = 2\pi\frac{\zeta^{|n|}}{r}
$$
for any $n \in \mathbb{Z}$.

Further we will solve the problem with respect to the parameter $\zeta$, not with respect to $\lambda$.
Let us show that $\lambda>0$ corresponds to $\zeta \in (0, 1)$.
First note that when $\lambda \in (0, +\infty)$  ${r \in (0, +\infty)}$ holds. Furthermore, solving the equation $r = \sqrt{\lambda(\lambda + 2\varkappa)}$ with respect to $\lambda$, we obtain $\lambda_{1,2} = -\varkappa \pm \sqrt{\varkappa^2 + r^2}$. Since we are interested in the case $\lambda > 0$, the equality $\lambda = -\varkappa + \sqrt{\varkappa^2 + r^2}$ holds. Then $\zeta$ can be written as $\zeta(r) = \frac{\sqrt{\varkappa^2 + r^2} - r}{\varkappa}$. In this case $\zeta(0) = 1, \lim\limits_{r \to \infty} \zeta(r) = \lim\limits_{r \to \infty} \frac{\sqrt{(\varkappa/r)^2 + 1} - 1}{\varkappa/r} = \lim\limits_{r \to \infty} \frac{(\varkappa/r)^2}{2\varkappa/r}  = \lim\limits_{r \to \infty} \frac{\varkappa}{2r} = 0$ and $\zeta'(r) = \frac{\frac{r}{\sqrt{\varkappa^2 + r^2}} - 1}{\varkappa} < 0$, i.e., the function $\zeta(r)$ is monotonically decreasing along $r$. From this we obtain that the condition $\lambda > 0$ corresponds to the condition $\zeta \in (0, 1)$.

Furthermore, note that the parameter $r$ satisfies the relation $r = \frac{\varkappa(1-\zeta^2)}{2\zeta}$. Thus,
\begin{align*}
  f(l) & = \beta f(0) \frac{\zeta^l}{r} - b_0 \sum\limits_{k=1}^{n} f(k) \frac{\zeta^{|k-l|}+\zeta^{|k+l|}}{r}                                                   \\
       & = \beta f(0) \frac{2\zeta^{l+1}}{\varkappa(1-\zeta^2)} - 2 b_0 \sum\limits_{k=1}^{n} f(k) \frac{\zeta^{|k-l|+1}+\zeta^{|k+l|+1}}{\varkappa(1-\zeta^2)},
\end{align*}

\begin{equation*}
  f(l) (\varkappa(1-\zeta^2)) = 2\beta f(0) \zeta^{l+1} - 2 b_0
  \sum\limits_{k=1}^{n} f(k) (\zeta^{|k-l|+1}+\zeta^{|k+l|+1}), \quad l = 0,\dots,n.
\end{equation*}

Let us write the matrix of the resulting system of equations of size $(n+1)\times(n+1)$, where the rows form the coefficients at $f(m)$, where $m$ varies from $0$ to $n$ in the $l$-th equation, $l = 0,\dots,n$.
\begin{equation*}
  \begin{pmatrix}
    2\beta\zeta - \varkappa(1-\zeta^2) & -4b_0\zeta^2                                 & \cdots & -4b_0\zeta^{n+1}                                \\
    2\beta\zeta^2                      & -2b_0\zeta(1+\zeta^2) - \varkappa(1-\zeta^2) & \cdots & -2b_0\zeta^n(1+\zeta^2)                         \\
    2\beta\zeta^3                      & -2b_0\zeta^2(1+\zeta^2)                      & \cdots & -2b_0\zeta^{n-1}(1+\zeta^4)                     \\
    \vdots                             & \vdots                                       & \ddots & \vdots                                          \\
    2\beta\zeta^{n+1}                  & -2b_0\zeta^n(1+\zeta^2)                      & \cdots & -2b_0\zeta(1+\zeta^{2n}) - \varkappa(1-\zeta^2)
  \end{pmatrix}.
\end{equation*}

Let us denote the determinant of this matrix by $\Delta(\zeta)$. If the relation $\Delta(\zeta) = 0$ is satisfied, there exists a nontrivial solution of the corresponding system of equations. We are interested in the case $\zeta \in (0,1)$. Thus, the problem of finding the condition for the existence of an isolated positive eigenvalue of the operator $\mathcal{H}_n$ reduces to the problem: when is there a solution of the equation $\Delta(\zeta) = 0$ for $\zeta \in (0,1)$?

Note that $\zeta = 1$ is a root of multiplicity $n$ of the equation $\Delta(\zeta) = 0$. This follows from the fact that when $\zeta = 1$ is substituted into $\Delta(\zeta)$, the rows of the corresponding matrix become equal. Then $\Delta(\zeta)$ can be written in the following form: $\Delta(\zeta) = (\zeta - 1)^{n}\Delta_1(\zeta)$ and $\Delta(\zeta)^{(n)}|_{\zeta = 1} = n!\Delta_1(\zeta)|_{\zeta=1}$.

Moreover, $\Delta(0) = (- \varkappa)^{n+1}$, whence $\Delta_1(0) = -\varkappa^{n+1}<0$. Hence, fulfillment of the condition $\Delta_1(1) > 0$ will mean that the equation $\Delta(\zeta) = 0$ has a root $\zeta \in (0,1)$. From the relation $\Delta(\zeta)^{(n)}|_{\zeta = 1} = n!\Delta_1(\zeta)|_{\zeta=1}$ it follows that $\Delta_1(1) = \frac{1}{n!}\Delta(\zeta)^{(n)}|_{\zeta = 1}$. Let's find $\Delta(\zeta)^{(n)}|_{\zeta = 1}$.

\begin{proposition} \label{st_J1}
  $\Delta(\zeta)^{(n)}|_{\zeta = 1} = 2^n J_n, $ where $J_n$ is the determinant of the following ${(n+2)\times(n+2)}$ matrix:
  \begin{equation*}
    (-1)^n
    \begin{pmatrix}
      0      & -2\beta             & -2b_0            & -2b_0            & -2b_0            & \cdots & -2b_0                 \\
      1      & 2\beta + 2\varkappa & 4b_0             & 6b_0             & 8b_0             & \cdots & 2(n+1)b_0             \\
      1      & 4\beta              & 4b_0 - \varkappa & 6b_0             & 8b_0             & \cdots & 2(n+1)b_0             \\
      1      & 6\beta              & 6b_0             & 6b_0 - \varkappa & 8b_0             & \cdots & 2(n+1)b_0             \\
      1      & 8\beta              & 8b_0             & 8b_0             & 8b_0 - \varkappa & \cdots & 2(n+1)b_0             \\
      \vdots & \vdots              & \vdots           & \vdots           & \vdots           & \ddots & \vdots                \\
      1      & 2(n+1)\beta         & 2(n+1)b_0        & 2(n+1)b_0        & 2(n+1)b_0        & \cdots & 2(n+1)b_0 - \varkappa
    \end{pmatrix}.
  \end{equation*}
\end{proposition}

\begin{proof}
  Note that by substituting $\zeta = 1$ any two columns of the matrix $\Delta(\zeta)$ become linearly dependent, whence it follows that only the $(n+1)$ summand remains when calculating the derivative: the $(n+1)$ determinant of matrices of size $(n+1)\times(n+1)$, in which only one column of the matrix $\Delta(\zeta)$ remains unchanged and the others are replaced by their derivatives.

  The determinant $\Delta(\zeta)$ at $\zeta = 1$ is as follows:
  \begin{equation*}
    \Delta(\zeta)|_{\zeta = 1}=
    \begin{vmatrix}
      2\beta & -4b_0  & \cdots & -4b_0  \\
      2\beta & -4b_0  & \cdots & -4b_0  \\
      2\beta & -4b_0  & \cdots & -4b_0  \\
      \vdots & \vdots & \ddots & \vdots \\
      2\beta & -4b_0  & \cdots & -4b_0
    \end{vmatrix}.
  \end{equation*}
  and the matrix consisting of the derivatives of  the columns $\Delta(\zeta)|_{\zeta = 1}$ at $\zeta = 1$ has the following form:
  \begin{equation*}
    \begin{vmatrix}
      2\beta + 2\varkappa & -8b_0              & \cdots & -4b_0(n+1)              \\
      4\beta              & -8b_0 + 2\varkappa & \cdots & -4b_0(n+1)              \\
      6\beta              & -12b_0             & \cdots & -4b_0(n+1)              \\
      \vdots              & \vdots             & \ddots & \vdots                  \\
      2(n+1)\beta         & -4b_0(n+1)         & \cdots & -4b_0(n+1) + 2\varkappa
    \end{vmatrix}.
  \end{equation*}
  So, $\Delta(\zeta)^{(n)}|_{\zeta = 1}$ is calculated by the formula:
  {\small\begin{equation*}
    \begin{aligned}
             & \begin{vmatrix}
                 2\beta + 2\varkappa & -8b_0              & \cdots & -4b_0  \\
                 4\beta              & -8b_0 + 2\varkappa & \cdots & -4b_0  \\
                 6\beta              & -12b_0             & \cdots & -4b_0  \\
                 \vdots              & \vdots             & \ddots & \vdots \\
                 2(n+1)\beta         & -4b_0(n+1)         & \cdots & -4b_0
               \end{vmatrix}
      + \cdots +
      \begin{vmatrix}
        2\beta + 2\varkappa & -4b_0  & \cdots & -4b_0(n+1)              \\
        4\beta              & -4b_0  & \cdots & -4b_0(n+1)              \\
        6\beta              & -4b_0  & \cdots & -4b_0(n+1)              \\
        \vdots              & \vdots & \ddots & \vdots                  \\
        2(n+1)\beta         & -4b_0  & \cdots & -4b_0(n+1) + 2\varkappa
      \end{vmatrix}
      \\[6pt]
      +      &
      \begin{vmatrix}
        2\beta & -8b_0              & \cdots & -4b_0(n+1)              \\
        2\beta & -8b_0 + 2\varkappa & \cdots & -4b_0(n+1)              \\
        2\beta & -12b_0             & \cdots & -4b_0(n+1)              \\
        \vdots & \vdots             & \ddots & \vdots                  \\
        2\beta & -4b_0(n+1)         & \cdots & -4b_0(n+1) + 2\varkappa
      \end{vmatrix}
      = -4b_0
      \begin{vmatrix}
        2\beta + 2\varkappa & -8b_0              & \cdots & 1      \\
        4\beta              & -8b_0 + 2\varkappa & \cdots & 1      \\
        6\beta              & -12b_0             & \cdots & 1      \\
        \vdots              & \vdots             & \ddots & \vdots \\
        2(n+1)\beta         & -4b_0(n+1)         & \cdots & 1
      \end{vmatrix}
      + \cdots                                                            \\[6pt]
      - 4b_0 &
      \begin{vmatrix}
        2\beta + 2\varkappa & 1      & \cdots & -4b_0(n+1)              \\
        4\beta              & 1      & \cdots & -4b_0(n+1)              \\
        6\beta              & 1      & \cdots & -4b_0(n+1)              \\
        \vdots              & \vdots & \ddots & \vdots                  \\
        2(n+1)\beta         & 1      & \cdots & -4b_0(n+1) + 2\varkappa
      \end{vmatrix}
      +  2\beta
      \begin{vmatrix}
        1      & -8b_0              & \cdots & -4b_0(n+1)              \\
        1      & -8b_0 + 2\varkappa & \cdots & -4b_0(n+1)              \\
        1      & -12b_0             & \cdots & -4b_0(n+1)              \\
        \vdots & \vdots             & \ddots & \vdots                  \\
        1      & -4b_0(n+1)         & \cdots & -4b_0(n+1) + 2\varkappa
      \end{vmatrix}.
    \end{aligned}
  \end{equation*}}

  The resulting sum can be assembled into a determinant of size $(n+2)\times(n+2)$ of the following form:
  \begin{equation*}
    \begin{vmatrix}
      0      & -2\beta             & 4b_0               & 4b_0                & 4b_0                & \cdots & 4b_0                    \\
      1      & 2\beta + 2\varkappa & -8b_0              & -12b_0              & -16b_0              & \cdots & -4(n+1)b_0              \\
      1      & 4\beta              & -8b_0 + 2\varkappa & -12b_0              & -16b_0              & \cdots & -4(n+1)b_0              \\
      1      & 6\beta              & -12b_0             & -12b_0 + 2\varkappa & -16b_0              & \cdots & -4(n+1)b_0              \\
      1      & 8\beta              & -16b_0             & -16b_0              & -16b_0 + 2\varkappa & \cdots & -4(n+1)b_0              \\
      \vdots & \vdots              & \vdots             & \vdots              & \vdots              & \ddots & \vdots                  \\
      1      & 2(n+1)\beta         & -4b_0(n+1)         & -4b_0(n+1)          & -4b_0(n+1)          & \cdots & -4b_0(n+1) + 2\varkappa
    \end{vmatrix}.
  \end{equation*}

  The decomposition of this determinant by the first row leads to the sum obtained above. Note that the determinant from the statement is obtained from this one by taking $-2$ from each column starting from the third column ($n$ times).
\end{proof}

Note that the size of the matrix under study has become larger, ${(n+2)\times(n+2)}$, but the matrix itself now has a simpler form.
Moreover, as was shown above, from the condition  $\Delta(\zeta)^{(n)}|_{\zeta = 1} > 0$ the existence of an isolated positive eigenvalue of the operator $\mathcal{H}_n$ follows. Whence we obtain the following statement.

\begin{proposition}
  Let $J_n > 0$, where $J_n$ is defined in Proposition~\ref{st_J1}. Then the operator $\mathcal{H}_n$ has an isolated positive eigenvalue that provides an exponential growth of particle numbers in the BRW with one particle generation center and $2n$ absorbing sources located symmetrically about the center.
\end{proposition}

Let's now calculate what the determinant $J_n$ is equal to.
First note that
\begin{align*}
  (-1)^nJ_n =-2^{n+1}\beta b_0^n
  \begin{vmatrix}
    0      & 1                   & 1                  & 1                  & \cdots & 1                      \\
    1      & 1 + \varkappa/\beta & 2                  & 3                  & \cdots & (n+1)                  \\
    1      & 2                   & 2 - \varkappa/2b_0 & 3                  & \cdots & (n+1)                  \\
    1      & 3                   & 3                  & 3 - \varkappa/2b_0 & \cdots & (n+1)                  \\
    \vdots & \vdots              & \vdots             & \vdots             & \ddots & \vdots                 \\
    1      & (n+1)               & (n+1)              & (n+1)              & \cdots & (n+1) - \varkappa/2b_0
  \end{vmatrix}.
\end{align*}

Let us first consider the determinant without the summand $-2^{n+1}\beta b_0^n$. We will subtract the rows one by one.
From the last row (numbered $(n+2)$) we subtract the second-to-last row (numbered $(n+1)$) and so on up to the second row, the last step: subtract the first row from the second row, leaving the first row unchanged. It is known that the determinant does not change with such transformations. After these steps we obtain the determinant of the following form:
\begin{equation*}
  \begin{vmatrix}
    0      & 1                   & 1                  & 1                  & 1               & \cdots & 1                & 1                \\
    1      & \varkappa/\beta     & 1                  & 2                  & 3               & \cdots & n-1              & n                \\
    0      & 1 - \varkappa/\beta & -\varkappa/2b_0    & 0                  & 0               & \cdots & 0                & 0                \\
    0      & 1                   & 1 + \varkappa/2b_0 & -\varkappa/2b_0    & 0               & \cdots & 0                & 0                \\
    0      & 1                   & 1                  & 1 + \varkappa/2b_0 & -\varkappa/2b_0 & \cdots & 0                & 0                \\
    \vdots & \vdots              & \vdots             & \vdots             & \vdots          & \ddots & \vdots           & \vdots           \\
    0      & 1                   & 1                  & 1                  & 1               & \cdots & - \varkappa/2b_0 & 0                \\
    0      & 1                   & 1                  & 1                  & 1               & \cdots & 1+\varkappa/2b_0 & - \varkappa/2b_0
  \end{vmatrix}
\end{equation*}

Let's decompose the resulting determinant by the first column, as a result we get the determinant of a matrix of size $(n+1)\times(n+1)$ of the form:
\begin{equation*}
  -
  \begin{vmatrix}
    1                   & 1                  & 1                  & 1               & \cdots & 1                & 1                \\
    1 - \varkappa/\beta & -\varkappa/2b_0    & 0                  & 0               & \cdots & 0                & 0                \\
    1                   & 1 + \varkappa/2b_0 & -\varkappa/2b_0    & 0               & \cdots & 0                & 0                \\
    1                   & 1                  & 1 + \varkappa/2b_0 & -\varkappa/2b_0 & \cdots & 0                & 0                \\
    \vdots              & \vdots             & \vdots             & \vdots          & \ddots & \vdots           & \vdots           \\
    1                   & 1                  & 1                  & 1               & \cdots & - \varkappa/2b_0 & 0                \\
    1                   & 1                  & 1                  & 1               & \cdots & 1+\varkappa/2b_0 & - \varkappa/2b_0
  \end{vmatrix}
\end{equation*}

Subtract the last row from the first row, then rearrange the rows so that the first row is the last ($(n+1)$-th), in this action we will rearrange the rows $n$ times: the first with the second, then the second with the third and so on up to the $n$-th row. After this transformation, the determinant will change sign $n$ times. The result is:
\begin{equation*}
  (-1)^{n+1}
  \begin{vmatrix}
    1 - \varkappa/\beta & -\varkappa/2b_0    & 0                  & 0               & \cdots & 0                & 0                  \\
    1                   & 1 + \varkappa/2b_0 & -\varkappa/2b_0    & 0               & \cdots & 0                & 0                  \\
    1                   & 1                  & 1 + \varkappa/2b_0 & -\varkappa/2b_0 & \cdots & 0                & 0                  \\
    \vdots              & \vdots             & \vdots             & \vdots          & \ddots & \vdots           & \vdots             \\
    1                   & 1                  & 1                  & 1               & \cdots & - \varkappa/2b_0 & 0                  \\
    1                   & 1                  & 1                  & 1               & \cdots & 1+\varkappa/2b_0 & - \varkappa/2b_0   \\
    0                   & 0                  & 0                  & 0               & \cdots & -\varkappa/2b_0  & 1 + \varkappa/2b_0
  \end{vmatrix}
\end{equation*}

Then we will subtract the rows as follows: from the row number $n$ (second-to-last) we subtract the row number $n-1$, then from $(n-1)$-th we subtract $(n-2)$-th, and so on up to the second row, the last step: from the second row we subtract the first one, and the first one we leave unchanged. As noted above, the determinant does not change after such transformations. As a result, we obtain the determinant of the tridiagonal matrix:
\begin{equation*}
  (-1)^{n+1}
  \begin{vmatrix}
    1 - \varkappa/\beta & -\varkappa/2b_0   & 0                 & 0               & \cdots & 0                & 0                  \\
    \varkappa/\beta     & 1 + \varkappa/b_0 & -\varkappa/2b_0   & 0               & \cdots & 0                & 0                  \\
    0                   & -\varkappa/2b_0   & 1 + \varkappa/b_0 & -\varkappa/2b_0 & \cdots & 0                & 0                  \\
    \vdots              & \vdots            & \vdots            & \vdots          & \ddots & \vdots           & \vdots             \\
    0                   & 0                 & 0                 & 0               & \cdots & - \varkappa/2b_0 & 0                  \\
    0                   & 0                 & 0                 & 0               & \cdots & 1+\varkappa/b_0  & - \varkappa/2b_0   \\
    0                   & 0                 & 0                 & 0               & \cdots & -\varkappa/2b_0  & 1 + \varkappa/2b_0
  \end{vmatrix}
\end{equation*}

Let us decompose the resulting determinant (so far ignoring $(-1)^{n+1}$) by the first column, we obtain the following expression with matrices of size $n\times n$:
\begin{align*}
  (1 - \varkappa/\beta)
   & \begin{vmatrix}
       1 + \varkappa/b_0 & -\varkappa/2b_0   & 0               & \cdots & 0                & 0                  \\
       -\varkappa/2b_0   & 1 + \varkappa/b_0 & -\varkappa/2b_0 & \cdots & 0                & 0                  \\
       \vdots            & \vdots            & \vdots          & \ddots & \vdots           & \vdots             \\
       0                 & 0                 & 0               & \cdots & - \varkappa/2b_0 & 0                  \\
       0                 & 0                 & 0               & \cdots & 1+\varkappa/b_0  & - \varkappa/2b_0   \\
       0                 & 0                 & 0               & \cdots & -\varkappa/2b_0  & 1 + \varkappa/2b_0
     \end{vmatrix} \\[6pt]
  -\varkappa/\beta
   & \begin{vmatrix}
       -\varkappa/2b_0 & 0                 & 0               & \cdots & 0                & 0                  \\
       -\varkappa/2b_0 & 1 + \varkappa/b_0 & -\varkappa/2b_0 & \cdots & 0                & 0                  \\
       \vdots          & \vdots            & \vdots          & \ddots & \vdots           & \vdots             \\
       0               & 0                 & 0               & \cdots & - \varkappa/2b_0 & 0                  \\
       0               & 0                 & 0               & \cdots & 1+\varkappa/b_0  & - \varkappa/2b_0   \\
       0               & 0                 & 0               & \cdots & -\varkappa/2b_0  & 1 + \varkappa/2b_0
     \end{vmatrix}
\end{align*}

Let us denote the first determinant by $I_n$. Note that the resulting expression is
\begin{equation*}
  (1 - \varkappa/\beta)I_n - \varkappa/\beta (-\varkappa/2b_0)I_{n-1}.
\end{equation*}
So, to find the determinant of interest, we are left to find $I_n$.
Doing a first-column expansion in $I_n$, we obtain a recurrence relation:
\begin{equation*}
  I_n = (1 + \varkappa/b_0)I_{n-1} + \varkappa/2b_0 (-\varkappa/2b_0)I_{n-2}.
\end{equation*}

The characteristic polynomial of this relation has the form
\begin{equation*}
  \lambda^2 - (1 + \varkappa/b_0)\lambda + \varkappa^2/4b_0^2 = 0.
\end{equation*}
Then the solutions to this equation are:
\begin{equation*}
  \lambda_{1,2} = \frac{(1 + \varkappa/b_0) \pm \sqrt{1 + 2\varkappa/b_0}}{2}.
\end{equation*}
Thus
\begin{equation*}
  I_n = c_1\lambda_1^n +c_2\lambda_2^n,
\end{equation*}
where the coefficients $c_{1,2}$ can be found from the relations on $I_{1,2}$:
\begin{equation*}
  I_1 = 1 + \frac{\varkappa}{2b_0}, \quad I_2 = 1 + \frac{3\varkappa}{2b_0} + \frac{\varkappa^2}{4b_0^2}.
\end{equation*}

Returning to the determinant $J_n$, we have for it an expression of the form:
\begin{equation*}
  \begin{aligned}
    J_n & = 2^{n+1}\beta b_0^n ((1 - \varkappa/\beta)I_n - \varkappa/\beta (-\varkappa/2b_0)I_{n-1}),                                                               \\
    J_n & = 2^{n+1}\beta b_0^n ((1 - \varkappa/\beta)(c_1\lambda_1^n +c_2\lambda_2^n) - \varkappa/\beta (-\varkappa/2b_0)(c_1\lambda_1^{n-1} +c_2\lambda_2^{n-1})).
  \end{aligned}
\end{equation*}

We are interested when $J_n > 0$, which is equivalent to the ratio:
\begin{equation*}
  \beta(1 - \varkappa/\beta)(c_1\lambda_1^n +c_2\lambda_2^n) - \varkappa (-\varkappa/2b_0)(c_1\lambda_1^{n-1} +c_2\lambda_2^{n-1}) > 0.
\end{equation*}

Thus, we can formulate the following theorem.

\begin{theorem} \label{th1_beta_cr}
  The inequality
  \begin{equation} \label{beta_c}
    \beta > \dfrac{\varkappa (c_1\lambda_1^n +c_2\lambda_2^n) - \frac{\varkappa^2}{2b_0}(c_1\lambda_1^{n-1} +c_2\lambda_2^{n-1})}{c_1\lambda_1^n +c_2\lambda_2^n},
  \end{equation}
  where $\lambda_{1,2}$ and $c_{1,2}$ are defined above, is a necessary and sufficient condition for the existence of an isolated positive eigenvalue of the operator $\mathcal{H}_n$, $n \in \mathbb{N}$.
\end{theorem}

The reasoning above provides sufficiency of the obtained condition. In order to show the necessity, we prove the following theorem.

\begin{theorem}
  The operator $\mathcal{H}_n$ can have at most one isolated positive eigenvalue. Moreover, this eigenvalue is simple.
\end{theorem}

\begin{proof}
  The proof of the simplicity of the eigenvalue is analogous to the corresponding reasoning from \cite{Y12_MZ:r,Y13-PS:e}.
  Recall that the operator $\mathcal{H}_n$ is given by the following formula
  \begin{equation*}
    \mathcal{H}_n = \varkappa \Delta + \beta\Delta_0 - b_0\sum\limits_{k=1}^{n} (\Delta_k + \Delta_{-k}).
  \end{equation*}

  Note that the operator $\varkappa \Delta$ is non-positively defined, that is, $\langle \varkappa \Delta u,u\rangle \le0$. This follows from the fact that its Fourier transform, as shown above, is nonpositive: $\widehat{\varkappa \Delta f}(\theta) = \varkappa(\cos{\theta} - 1)\hat{f}(\theta)$.
  The same relation is true for the operator $\sum_{k=1}^{n} - b_0(\Delta_k + \Delta_{-k})$.

  From this we conclude that
  \begin{equation*}
    \left\langle \Biggl(\varkappa \Delta-\sum\limits_{k=1}^{n} b_0(\Delta_k + \Delta_{-k})\Biggr) u,u\right\rangle \le0
  \end{equation*}
  and, hence, the spectrum of the operator
  \begin{equation*}
    \varkappa \Delta-b_0 \sum\limits_{k=1}^{n} (\Delta_k + \Delta_{-k})
  \end{equation*}
  lies on the non-positive part of the real axis, in particular, the operator
  \begin{equation}\label{E-xxx}
    \varkappa \Delta-\lambda I -b_0 \sum\limits_{k=1}^{n} (\Delta_k + \Delta_{-k})
  \end{equation}
  is reversible for every $\lambda>0$.

  To investigate the problem on eigenvalues of the operator $\mathcal{H}_n$, consider the equation
  \begin{equation*}\label{E-eigenvec2}
    \lambda f= \varkappa \Delta f -b_0 \sum\limits_{k=1}^{n} (\Delta_k + \Delta_{-k})f + \beta\Delta_0 f,
  \end{equation*}
  which can be rewritten as
  \begin{equation*}\label{E-eigenvec3}
    \varkappa \Delta f - \lambda f -b_0 \sum\limits_{k=1}^{n} (\Delta_k + \Delta_{-k})f + \beta\Delta_0 f=0.
  \end{equation*}

  Hence, due to the already proved reversibility of the operator \eqref{E-xxx}, we obtain that the last equation is equivalent to the following one:
  \begin{equation*}
    f =- \left(	\varkappa \Delta - \lambda I -b_0 \sum\limits_{k=1}^{n} (\Delta_k + \Delta_{-k})\right)^{-1}
    \beta\Delta_0 f,
  \end{equation*}
  in which $\beta\Delta_0 f = \beta\langle \delta_{0},f\rangle \delta_{0}$.

  Thus, the dimension of the subspace of eigenvectors corresponding to the eigenvalue $\lambda>0$ of the operator $\mathcal{H}_n$ does not exceed one, which means that the eigenvalues are simple.

  Let us show that the isolated eigenvalue of the operator $\mathcal{H}_n$ is unique. Suppose that $\lambda_1 \neq \lambda_2$ are isolated positive eigenvalues, and $f_1$ and $f_2$ are eigenfunctions corresponding to these values, then $\langle f_1, f_2 \rangle = 0$. Without restriction of generality, we assume that $\lambda_2>\lambda_1$, and $\langle \delta_0, f_1 \rangle = \langle \delta_0, f_2 \rangle$ (since the eigenfunctions are defined up to a constant).

  Let us introduce the following notation
  \begin{equation*}
    R_{\lambda_i} = \left(\varkappa \Delta - \lambda_i I -b_0 \sum\limits_{k=1}^{n} (\Delta_k + \Delta_{-k})\right)^{-1}, \quad i=1,2.
  \end{equation*}
  Then we have
  \begin{equation*}
    \begin{cases}
      f_1 = -R_{\lambda_1} \beta \Delta_0 f_1, \\
      f_2 = -R_{\lambda_2} \beta \Delta_0 f_2, \\
    \end{cases}
  \end{equation*}
  whence
  \begin{equation*}
    R_{\lambda_2}f_1 = -R_{\lambda_2}R_{\lambda_1} \beta \Delta_0 f_1 = R_{\lambda_1}(-R_{\lambda_2} \beta \Delta_0 f_2) = R_{\lambda_1}f_2.
  \end{equation*}

  From here we obtain the following expression
  \begin{align*}
    f_1 & = R_{\lambda_2}^{-1}R_{\lambda_1}f_2 =  \left(\varkappa \Delta -b_0 \sum\limits_{k=1}^{n} (\Delta_k + \Delta_{-k}) - \lambda_1 I + (\lambda_1 - \lambda_2)I\right)R_{\lambda_1}f_2 \\
        & = (R_{\lambda_1}^{-1} + (\lambda_1 - \lambda_2)I)R_{\lambda_1}f_2 =
    (I + (\lambda_1 - \lambda_2)R_{\lambda_1})f_2.
  \end{align*}
  Multiplying the obtained expression scalarly by $f_2$, we have
  \begin{align*}
    \langle f_1,f_2 \rangle & = \langle f_2 + (\lambda_1 - \lambda_2)R_{\lambda_1}f_2, f_2 \rangle,                       \\
    \langle f_1,f_2 \rangle & = \langle f_2, f_2 \rangle + (\lambda_1 - \lambda_2) \langle R_{\lambda_1}f_2, f_2 \rangle.
  \end{align*}

  Thus $\langle f_2, f_2 \rangle > 0$, $(\lambda_1 - \lambda_2) < 0$ by assumption, and $\langle R_{\lambda_1}f_2, f_2 \rangle \rangle \le 0$ due to non-positive definiteness of $R_{\lambda_1}^{-1}$, while $\langle f_1,f_2 \rangle = 0$. The obtained contradiction proves the uniqueness of the positive isolated eigenvalue of the operator $\mathcal{H}_n$.
\end{proof}

So, the existence of the root $\zeta \in (0,1)$ of the equation $\Delta(\zeta) = 0$ provides the fulfillment of the condition $\Delta_1(1) > 0$, which implies the fact that the condition obtained in Theorem~\ref{th1_beta_cr} is a necessary and sufficient condition for the exponential growth of particle numbers in the considered process.

\section{BRW with an infinite number of absorbing sources}\label{sec3}

In addition, the following theorem has been proved.

\begin{theorem}\label{th2_beta_cr}
  For $n \rightarrow \infty$, the condition
  \begin{equation*}
    \beta > \dfrac{\varkappa (c_1\lambda_1^n +c_2\lambda_2^n) - \frac{\varkappa^2}{2b_0}(c_1\lambda_1^{n-1} +c_2\lambda_2^{n-1})}{c_1\lambda_1^n +c_2\lambda_2^n},
  \end{equation*}
  providing exponential growth of particle numbers in the process with one particle generation center and $2n$ symmetric absorbing sources, has the form
  \begin{equation*}
    \beta^* > b_0\sqrt{1 + 2\varkappa/b_0}.
  \end{equation*}
\end{theorem}

\begin{proof}
  Recall that
  \begin{equation*}
    \lambda_{1,2} = \frac{(1 + \varkappa/b_0) \pm \sqrt{1 + 2\varkappa/b_0}}{2}.
  \end{equation*}
  Note that $0 < \lambda_2 < 1$. Indeed, denote $x = \varkappa/b_0$ and consider the function $f(x) = 1 + x - \sqrt{1+2x}$. It is easy to show that $f'(x) = 0$ when $x = 1 \pm \sqrt{2}$, with the function increasing on the segment $[1-\sqrt{2},1+\sqrt{2}]$ and decreasing on $[1+\sqrt{2}, +\infty)$, that is, the point $1+\sqrt{2}$ is the maximum point of $f(x)$. Whence we obtain that $f(x) \leq f(1+\sqrt{2}) = 2 +\sqrt{2} - \sqrt{3+2\sqrt{2}} \leq 2 +\sqrt{2} - \sqrt{2+\sqrt{2}} < 2$, or, which is the same, $\lambda_2<1$. The relation $\lambda_2 > 0$ follows, for example, from the fact that $f(x) = \frac{x^2}{1+x+\sqrt{1+2x}} > 0$ when $x>0$.

  Furthermore, $\lambda_1 > 1$. Indeed, let us put, as above, $x = \varkappa/b_0$, then $\lambda_1(x) = g(x)/2$, where $g(x) = 1+x+\sqrt{1+2x} > 2$ at $x>0$, and hence $\lambda_1 > 1$.
  Therefore, when $n \rightarrow \infty$ $\lambda_2^n \rightarrow 0$, and the ratio under study (after reducing the numerator and denominator by $c_1 \lambda_1^{n-1}$) takes the form:
  \begin{equation*}
    \beta > \dfrac{\varkappa \lambda_1 - \frac{\varkappa^2}{2b_0}}{\lambda_1}.
  \end{equation*}

  Recall that $\lambda_{1,2}$ are the roots of the quadratic equation
  \begin{equation*}
    \lambda^2 - (1 + \varkappa/b_0)\lambda + \varkappa^2/4b_0^2 = 0.
  \end{equation*}
  So, by Vyet's theorem, $\frac{1}{\lambda_1} = \frac{\lambda_2}{\varkappa^2/4b_0^2}$, whence we get:

  \begin{equation*}
    \beta > \varkappa - 2b_0\lambda_2
  \end{equation*}

  \begin{equation*}
    \beta >  b_0(\sqrt{1 + 2\varkappa/b_0} - 1) , \quad \beta^* > b_0\sqrt{1 + 2\varkappa/b_0}.
  \end{equation*}
\end{proof}

Note, that the obtained condition coincides with the condition providing an exponential growth of particle numbers in the process with one particle generation center and an infinite number of absorbing sources located at every point, which can be obtained using the results of \cite{art}.
Indeed, it was shown in \cite{art} that the isolated eigenvalue $\lambda_{\infty}$ of the evolutionary operator
\begin{equation*}
  \mathcal{H}_{\infty} = \mathcal{A} + \beta^*\Delta_0 - b_0 I,
\end{equation*}
where $I$ is the identity operator, has the form $\lambda_{\infty} = \lambda_0 - b_0$. Here $\lambda_0$ is the isolated positive eigenvalue of the operator $\mathcal{A} + \beta^*\Delta_0$.
In the case where the walking operator $\mathcal{A}$ is the operator $\varkappa\Delta$, by solving the eigenvalue problem using the Fourier transform, similar to the reasoning above, we can show that $\lambda_0 = \sqrt{\varkappa^2 + (\beta^*)^2} - \varkappa$,  that is, $\lambda_{\infty} = \sqrt{\varkappa^2 + (\beta^*)^2} - \varkappa - b_0$, from which it is easy to obtain that the condition $\lambda_{\infty} > 0$ is equivalent to the condition $\beta^* > b_0\sqrt{1 + 2\varkappa/b_0}$.

Let us consider how the critical value for $\beta^*$, above which an isolated positive eigenvalue appears in the spectrum of the operator $\mathcal{H}_n$, behaves depending on $n$, that is the number of absorbing sources located symmetrically to the right and to the left of the particle generation center (the total number of absorbing sources is $2n$).

The figure below plots such values for $\beta^*$ as a function of $n$ for fixed values of the parameters $\varkappa$ and $b_0$.

\begin{figure}[ht]%
  \centering
  \includegraphics[width=0.9\textwidth]{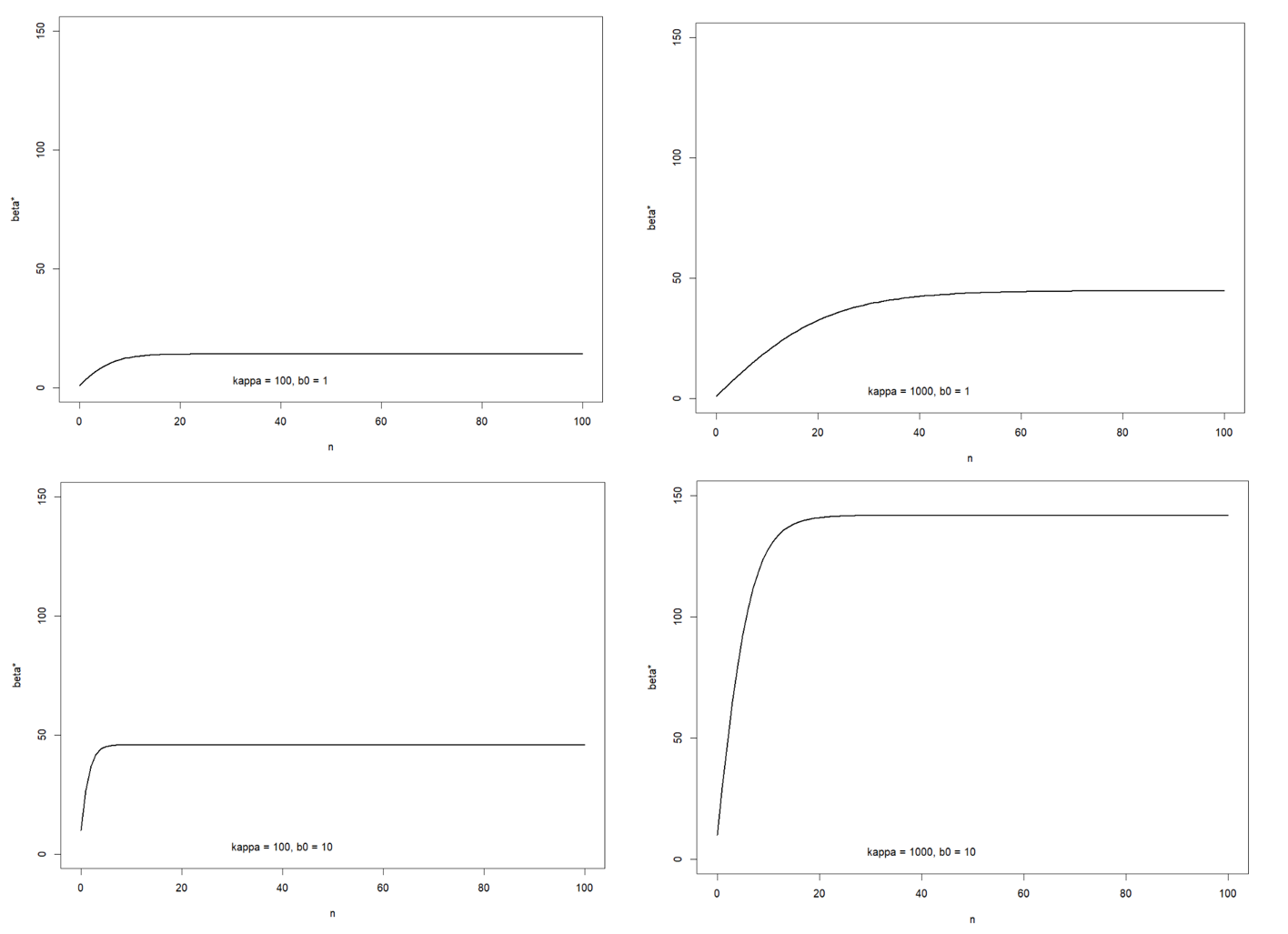}
  \caption{The critical value for $\beta^*$ as a function of $n$, which provides an exponential growth of particle numbers in a process with $2n$ absorbing sources, for various values of the parameters $\varkappa$ and $b_0$.}\label{fig1}
\end{figure}

The value of the absorption intensity $b_0$ increases along the rows and the value of the walking intensity $\varkappa$ increases along the columns. It can be seen that the convergence to the limit value is quite fast.

\section{Conclusion}\label{sec4}

In this paper we have obtained conditions (in terms of model parameters) which lead to exponential growth of particle numbers in the process with one particle generation center and $2n$ symmetric absorbing sources. The obtained conditions allow us to answer the question: how large should be the intensity of the particle generation center ($\beta^*$ or $\beta$) with respect to the absorption intensity $b_0$ and the walking intensity $\varkappa$ for the process to grow exponentially? In the case of considering a simple random walk along $\mathbb{Z}$, we succeeded in obtaining exact solutions to the problem for both an arbitrary finite $n$, in Theorem \ref{th1_beta_cr}, and for the case when the absorbing sources are located at every point of $\mathbb{Z}$, in Theorem \ref{th2_beta_cr}.
Moreover, the conditions obtained in the paper are necessary and sufficient.


\end{document}